\documentclass[11pt]{amsart}
\usepackage{amssymb}
\usepackage{amsmath}
\usepackage{color}
\usepackage{graphicx}
\usepackage{epstopdf}
\usepackage[top=1in, bottom=1in, left=1.25in, right=1.25in]{geometry}

\newtheorem{thm}{Theorem}[section]
\newtheorem{lemma}[thm]{Lemma}
\newtheorem{cor}[thm]{Corollary}

\theoremstyle{definition}

\theoremstyle{remark}
\newtheorem{remark}[thm]{Remark}

\theoremstyle{rem}

\newcommand\bq{\begin{equation}}
\newcommand\eq{\end{equation}}
\newcommand\beq{\begin{eqnarray*}}
\newcommand\eeq{\end{eqnarray*}}
\newcommand\ben{\begin{enumerate}}
\newcommand\een{\end{enumerate}}
\newcommand\bit{\begin{itemize}}
\newcommand\eit{\end{itemize}}

\newcommand\des{{\rm des}}
\newcommand\desB{{\rm des_B}}
\newcommand\DesB{{\rm DES_B}}
\newcommand\desD{{\rm des_D}}
\newcommand\DesD{{\rm DES_D}}

\newcommand\Des{{\rm DES}}

\newcommand\mdB{{\rm maxdrop_B}}
\newcommand\st{{\rm st}}
\newcommand\sst{{\rm sst}}

\def\til{\widetilde}

\keywords{Eulerian polynomials, descent number, type B Coxeter group, signed permutations, type D Coxeter group, even signed permutations, interlacing roots, compatible polynomials}


\begin{document}

\title[Recurrences for Eulerian poly. of type B and D]{Recurrences for Eulerian polynomials of type B and type D}
\date{January 22, 2015}
\author[M. Hyatt]{Matthew Hyatt}
\address{Department of Mathematics, Lehigh University, Bethlehem, PA 18015}
\address{Mathematics Department, Pace University, Pleasantville, NY 10570}
\email{mhyatt@pace.edu}

\begin{abstract}

We introduce new recurrences for the type B and type D Eulerian polynomials,
and interpret them combinatorially. These recurrences
are analogous to a well-known recurrence for the type A Eulerian polynomials.
We also discuss their relationship to polynomials introduced by Savage and Visontai
in connection to the real-rootedness of the corresponding Eulerian polynomials.


\end{abstract}

\maketitle

\vbox{\tableofcontents}

\section{Introduction}\label{section intro}

Let $S_n$ denote the group of permutations on the set $[n]:=\{1,2,\dots , n\}$.
Let $A_n(t)$ denote the classic (type A) Eulerian polynomial, that is
\[A_n(t)=\sum_{\pi\in S_n}t^{\des(\pi)},\]
where the descent number of $\pi\in S_n$ is defined by
\[\des(\pi)=|\{i\in [n-1]:\pi(i)>\pi(i+1)\}|.\]
These polynomials were first introduced by Euler, although he did not
define them via descents of permutations (see \cite{Foata}). The following generating function identity, which is 
attributed to Euler, is an equivalent definition of these polynomials
\[\sum_{n\geq 0}A_n(t)\frac{x^n}{n!}=\frac{t-1}{t-\exp((t-1)x)}.\]
There are numerous recurrences for the type A Eulerian polynomials.
For the purposes of this paper we call the reader's attention to
the following recurrence which holds for $n\geq 1$, and is also attributed to Euler
\begin{equation}\label{Arec}
A_n(t)=\sum_{k=0}^{n-1}A_k(t){n \choose k}(t-1)^{n-k-1}.
\end{equation}
The main goal of this paper is to find recurrences analogous to \eqref{Arec} 
for the type B and type D Eulerian polynomials,
and interpret them combinatorially.


In section \ref{sectionB} we establish notation for the group of signed permutations, denoted $B_n$. This is
a Coxeter group of type B, and thus has an analogous notion of descent. We let
\[B_n(t)=\sum_{\pi\in B_n}t^{\desB(\pi)},\]
which are called the type B Eulerian polynomials.
While there are several recurrences for these polynomials
(see for example \cite{Brenti}, \cite{Chow2}, \cite{ChoGes}, \cite{Stembridge}),
we give a new recurrence which 
we consider to be an analog of \eqref{Arec}.
\begin{thm}\label{thmB}
Let $\displaystyle P_n(t)=\sum_{k=0}^{n-1}B_k(t){n \choose k}(t-1)^{n-k-1}$. For $n\geq 1$ we have
\[B_n(t)=P_n(t)+t^nP_n(t^{-1}).\]
\end{thm}
\begin{remark}
Given a polynomial $f(t)=\displaystyle\sum_{k=0}^nc_kt^k$ of degree $n$, the reverse of $f$, denoted $\til{f}$, 
is given by $\displaystyle \til{f}(t)=t^nf(t^{-1})=\sum_{k=0}^nc_{n-k}x^k.$
Since $P_n(t)$ has degree $n-1$, Theorem \ref{thmB} may be restated as
\[B_n(t)=P_n(t)+t\widetilde{P}_n(t).\] 
\end{remark}
Theorem \ref{thmB} is new in the sense that it does not explicitly appear in the literature.
However we will show how this theorem can be deduced from a well known generating function identity.
We will, in addition, provide a combinatorial proof of Theorem \ref{thmB}.

A recurrence similar to Theorem \ref{thmB} also holds for the type D case. 
We begin section \ref{sectionD} by establishing 
notation for the group of even signed permutations, denoted $D_n$. This is
a Coxeter group of type D, and thus has an analogous notion of descent. We let
\[D_n(t)=\sum_{\pi\in D_n}t^{\desD(\pi)},\]
which are called the type D Eulerian polynomials. 
Again, there are several recurrences for these polynomials
(see for example \cite{Brenti}, \cite{Chow2003}, \cite{Chow2}, \cite{Stembridge}).
Here we give a new recurrence, which is
analogous to \eqref{Arec} and Theorem \ref{thmB}.
\begin{thm}\label{thmD}
Let $\displaystyle Q_n(t)=\sum_{k=0}^{n-1}D_k(t){n \choose k}(t-1)^{n-k-1}$. For $n\geq 2$ we have
\[D_n(t)=Q_n(t)+t^nQ_n(t^{-1})=Q_n(t)+t\widetilde{Q}_n(t).\]
\end{thm}
While there is a well known generating function identity for the type D Eulerian polynomials, 
we see no easy way to deduce Theorem \ref{thmD} from this identity. Instead we provide a combinatorial
proof of Theorem \ref{thmD}.

It turns out that certain refinements of the polynomials $Q_n(t), t^nQ(t^{-1}), P_n(t)$, and $t^nP(t^{-1})$ 
were introduced by Savage and Visontai in \cite{SavVis}, where they proved Brenti's \cite{Brenti} conjecture
that the type D Eulerian polynomials have only real roots. We discuss this connection in section \ref{real}.

\section{A recurrence for the Eulerian polynomials of type B}\label{sectionB}

As mentioned above, we let $B_n$ denote the group of signed permutations. An element
$\pi\in B_n$ is a bijection on the integers $[-n,n]$ such that $\pi(-i)=-\pi(i)$, in particular
$\pi(0)=0$. We will often 
write elements in the "window" notation, i.e. $\pi=[\pi(1),\pi(2),\dots ,\pi(n)]$, and we will also
use cycle notation. $B_n$ is a type B Coxeter group with generators 
$\tau_0,\tau_1,\dots ,\tau_{n-1}$ where
$\tau_0=[-1,2,3,\dots , n]=(1,-1)$, and $\tau_j=(j, j+1)$ for $j=1,\dots ,n-1$.
Given $\pi\in B_n$ define
\[\DesB(\pi)=\{i\in [0,n-1]:\pi(i)>\pi(i+1)\},\]
and
\[\desB(\pi)=|\DesB(\pi)|.\]
This definition coincides with the notion of Coxeter descent 
(i.e. $i\in\DesB(\pi)$ if and only if the Coxeter length of $\pi\tau_i$ is less than 
the Coxeter length of $\pi$).
The type B Eulerian polynomials, $B_n(t)$, satisfy the identity \cite[Prop 7.1 (b)]{Stembridge}
\[\frac{B_n(t)}{(1-t)^{n+1}}=\sum_{k\geq 0}(2k+1)^nt^k,\]
from which one can deduce (see also \cite[Theorem 3.4 (iv) with $q=1$]{Brenti})
\begin{equation}\label{Bgen}
\sum_{n\geq 0}B_n(t)\frac{x^n}{n!}=\frac{(1-t)\exp(x(1-t))}{1-t\exp(2x(1-t))}.
\end{equation}


Next we show how Theorem \ref{thmB} can be deduced from \eqref{Bgen}.

\

\hspace{-12pt}\textbf{\textit{First Proof of Theorem \ref{thmB}.}}
\begin{proof}

From \eqref{Bgen} we have
\[\left(\sum_{n\geq 0}B_n(t)\frac{x^n}{n!}\right)
\left(\frac{\exp(-x(1-t))-t\exp(x(1-t))}{(1-t)}\right)=1,\]
\[\left(\sum_{n\geq 0}B_n(t)\frac{x^n}{n!}\right)
\left(\sum_{n\geq 0}\frac{x^n}{n!}(1-t)^{n-1}((-1)^n-t)\right)=1,\]
\[\sum_{n\geq 0}\frac{x^n}{n!}\sum_{k=0}^{n}{n \choose k}B_k(t)(1-t)^{n-k-1}((-1)^{n-k}-t)=1.\]
So for $n\geq 1$ we have
\[\sum_{k=0}^{n}{n \choose k}B_k(t)(1-t)^{n-k-1}((-1)^{n-k}-t)=0,\]
which implies
\begin{align*}
B_n(t)&=\sum_{k=0}^{n-1}{n \choose k}B_k(t)(1-t)^{n-k-1}(t+(-1)^{n-k-1}) \\
&=P_n(t)+t\sum_{k=0}^{n-1}{n \choose k}B_k(t)(1-t)^{n-k-1} \\
&=P_n(t)+t\sum_{k=0}^{n-1}{n \choose k}t^kB_k(t^{-1})t^{n-k-1}(t^{-1}-1)^{n-k-1}\\
&=P_n(t)+t^nP_n(t^{-1}).
\end{align*}
The penultimate step uses the fact that $B_k(t)$ is symmetric of degree $k$ (see \cite[Theorem 2.4]{Brenti}).


\end{proof}

Next we provide a combinatorial proof of Theorem \ref{thmB}. We begin with two lemmas, and then show 
they imply the theorem.

\begin{lemma}\label{lemB1}
Define $B_n^+=\{\pi\in B_n:\pi(n)>0\}$. Then for $n\geq 1$
\[\sum_{\pi\in B_n^+}t^{\desB(\pi)}=P_n(t)=\sum_{k=0}^{n-1}{n \choose k}B_k(t)(t-1)^{n-k-1}.\]
\end{lemma}

\begin{proof}
Following previous work in \cite{Hyatt2}, define
\[\mdB(\pi)=\max\Big\{\max\{i-\pi(i):\pi(i)>0\},\max\{i:\pi(i)<0\}\Big\},\]
and
\[B_{n}^{(j)}(t)=\sum_{\substack{\pi\in B_n \\ \mdB(\pi)\leq j}}t^{\desB(\pi)}.\]
Although the definition of $\mdB$ is motivated by the so-called type B bubble sort, notice that
$B_n^+=\{\pi\in B_n:\mdB(\pi)\leq n-1\}$. Thus 
\[\sum_{\pi\in B_n^+}t^{\desB(\pi)}=B_{n}^{(n-1)}(t).\]
It is proved (combinatorially) that for $m\geq 0$ we have \cite[Theorem 1.6]{Hyatt2}
\[B_{m+j+1}^{(j)}(t)=\sum_{k=1}^{j+1}{j+1 \choose k}B_{m+j+1-k}^{(j)}(t)(t-1)^{k-1}.\]
Setting $m=0$ and replacing $j$ with $n-1$ we obtain
\begin{align*}
B_{n}^{(n-1)}(t)&=\sum_{k=1}^{n}{n \choose k}B_{n-k}^{(n-1)}(t)(t-1)^{k-1} \\
&=\sum_{k=1}^{n}{n \choose k}B_{n-k}(t)(t-1)^{k-1} \\
&=\sum_{k=0}^{n-1}{n \choose k}B_{k}(t)(t-1)^{n-k-1},
\end{align*}
where the penultimate step follows from the fact that $B_{n}^{(j)}(t)=B_{n}(t)$ whenever $n\leq j$.

\end{proof}

\begin{lemma}\label{lemB2}
Define $B_n^-=\{\pi\in B_n:\pi(n)<0\}$. Then for $n\geq 1$
\[\sum_{\pi\in B_n^-}t^{\desB(\pi)}=t^nP_n(t^{-1}).\]
\end{lemma}

\begin{proof}

Define a bijection $\phi:B_n^+\rightarrow B_n^-$ by $\phi(\pi(i))=-\pi(i)$ for $i=1,\dots ,n$.
Then $\pi(1)<0$ if and only if $\phi(\pi(1))>0$. And for $i=1,2,\dots ,n-1$ we have $\pi(i)>\pi(i+1)$ if and only if $\phi(\pi(i))<\phi(\pi(i))$.
Thus $\desB(\pi)+\desB(\phi(\pi))=n$. Then
\[\sum_{\pi\in B_n^-}t^{\desB(\pi)}=\sum_{\pi\in B_n^+}t^{\desB(\phi(\pi))}
=\sum_{\pi\in B_n^+}t^{n-\desB(\pi)}
=t^n\sum_{\pi\in B_n^+}t^{-\desB(\pi)}
=t^nP_n(t^{-1}).\]

\end{proof}

\

\hspace{-12pt}\textbf{\textit{Second Proof of Theorem \ref{thmB}.}}
\begin{proof}

Using Lemmas \ref{lemB1} and \ref{lemB2} we have
\[B_n(t)=\sum_{\pi\in B_n^+}t^{\desB(\pi)}+\sum_{\pi\in B_n^-}t^{\desB(\pi)}
=P_n(t)+t^nP_n(t^{-1}).\]

\end{proof}

\section{A recurrence for the Eulerian polynomials of type D}\label{sectionD}

Let $D_n$ denote the group of even signed permutations. We call $\pi$ an 
even signed permutation if $\pi$ is a signed permutation such that there are an even number of 
negative letters among $\pi(1),\pi(2),\dots ,\pi(n)$. $D_n$ is a Coxeter group of type D with 
generators $\tau_0,\tau_1,\dots, \tau_{n-1}$ where $\tau_0=[-2, -1, 3, 4, 5, \dots ,n]=(1,-2)$, 
and $\tau_j=(j, j+1)$ for $j=1,\dots ,n-1$. Given $\pi\in D_n$ define
\[\Des(\pi)=\{i\in[n-1]:\pi(i)\geq \pi(i+1)\},\]
\[\DesD(\pi)=\begin{cases}
\Des(\pi) & \text{if } \pi(1)+\pi(2)>0 \\
\Des(\pi)\cup\{0\} & \text{if }\pi(1)+\pi(2)<0 \end{cases},\]
\[\desD(\pi)=|\DesD(\pi)|.\]
As in the type B case, this definition is motivated by the fact that it coincides with the notion of 
type D Coxeter descent.
The type D Eulerian polynomials, $D_n(t)$, satisfy
the following generating function formula due to Brenti \cite{Brenti}
\begin{equation}\label{Dgen}
\sum_{n\geq 0}D_n(t)\frac{x^n}{n!}=\frac{(1-t)\exp(x(1-t))-xt(1-t)\exp(2x(1-t))}{1-t\exp(2x(1-t))}.
\end{equation}

Although we can obtain a recurrence from \eqref{Dgen} by an approach similar to that of the first proof of 
Theorem \ref{thmB}, doing so yields the following somewhat unpleasant formula
\[D_n(t)=\sum_{k=0}^{n-1}{n\choose k}D_k(t)\sum_{j=0}^{n-k}
\frac{(n-k)!}{j!}t^{n-j-k}(1-t)^{j-1}\left[t(n-j-k+1)^j-(n-j-k-1)^j\right].\]


Instead we will prove Theorem \ref{thmD} in a manner similar to that of the second proof of Theorem \ref{thmB}. 
The first step is to
obtain a type D analog of Lemma \ref{lemB1}, which we accomplish by adapting the methods from
\cite{Hyatt2}, which in turn are extensions of methods of Chung, Claesson, Dukes, and Graham \cite{ccdg}.

We begin by discussing standardization of permutations. Suppose we have a finite set 
$C=\{c_1,c_2,\dots,c_n\}\subset \mathbb{N}$ with $c_1<c_2<\dots<c_n$, and a 
permutation $\pi$ on $C$. The \textit{standardization} of $\pi$, denoted $\st(\pi)$, is the 
permutation in $S_n$ obtained from $\pi$ by replacing $c_i$ with $i$. For example 
$\st([4,5,2,9,7])=[2,3,1,5,4]$. Given a signed permutation $\pi$, let 
$|\pi|=[|\pi(1)|,|\pi(2)|,\dots ,|\pi(n)|].$ We call $\pi$ a signed permutation on the set $C$, 
if $\pi$ is a word over $\mathbb{Z}$ and $|\pi|$ is permutation on $C$. 
We define the \textit{signed standardization} of a signed permutation, denoted $\sst(\pi)$, by
\[\sst(\pi)(i)=\begin{cases}
\st(|\pi|)(i) & \text{ if }\pi(i)>0 \\
-\st(|\pi|)(i) & \text{ if }\pi(i)<0
\end{cases}\]
For example $\sst([-6,4,-2,9,7])=[-3,2,-1,5,4]$. If the set $C$ is fixed, then the inverse of $\sst$, 
denoted $\sst_C^{-1}$, is well-defined. For example if $C=\{2,4,6,7,9\}$ and $\pi=[-3,2,-1,5,4]$, 
then $\sst_C^{-1}(\pi)=[-6,4,-2,9,7]$. If we extend the definition of type D descent set for any word 
over $\mathbb{Z}$, then it is clear that both $\sst$ and $\sst_C^{-1}$ preserve the type D descent 
set of a word. For example $\DesD([-6,4,-2,9,7])=\{0,2,4\}=\DesD([-3,2,-1,5,4])$.

Given $S\subseteq [0,n-1]$ and $U\subseteq D_n$, define
\[U(S)=\{\pi\in U:\DesD(\pi)\supseteq S\},\]
\[r_n(S)=\max\{i:[n-i,n-1]\subseteq S\},\]
with $r_n(S)=0$ if $n-1\notin S$. For our purposes $U$ will either be $D_n^+$, which we define by
$D_n^+=\{\pi\in D_n:\pi(n)>0\}$, or $U$ will be all of $D_n$. Let $\displaystyle {[n]\choose j}$ denote
the set of $j$-element subsets of $[n]$.
Given $\pi\in D_n^+(S)$, define a map $\psi$ by
\[\psi(\pi)=(\sigma,X),\]
where
\[\sigma=\sst([\pi(1),\pi(2),\dots ,\pi(n-i-1)]),\]
and
\[X=\{\pi(n-i),\pi(n-i+1),\dots ,\pi(n)\},\]
where $i=r_n(S)$.

For example consider $S=\{1,6,7\}$ and $\pi=[2, -3, 5, 1, -8, 7, 6, 4]\in D_8^+(S)$. Then $r_8(S)=2$
and $\psi(\pi)=(\sigma,X)$ where $\sigma=[2,-3,4,1,-5]$ and $X=\{4,6,7\}$.

\begin{lemma}\label{lemD0}
The map $\psi:D_n^+(S)\rightarrow D_{n-i-1}(S\cap[0,n-i-2])\times{[n]\choose i+1}$ is a bijection,
where $i=r_n(S)$. Consequently $|D_n^+(S)|=|D_{n-i-1}(S\cap[0,n-i-2])|{n\choose i+1}$.
\end{lemma}

\begin{proof}

Given $\pi\in D_n^+(S)$, let $\psi(\pi)=(\sigma, X)$ as above. Since signed standardization preserves type D
descent set, $\sigma\in D_{n-i-1}(S\cap[0,n-i-2])$. Since $[n-i,n-1]\subseteq \DesD(\pi)$ and $\pi(n)>0$, we have
$X\in {[n]\choose i+1}$. Thus $\psi$ is well-defined.

Next we describe the inverse map. Let $\sigma\in D_{n-i-1}(T)$ where $T\subseteq [0,n-i-2]$, and let 
$X=\{x_1,x_2,\dots ,x_{i+1}\}\in {[n]\choose i+1}$ where $x_1<x_2< \dots <x_{i+1}$. Then set
\[\psi^{-1}(\sigma,X)=\sst_{[n]\setminus X}^{-1}(\sigma)*(x_{i+1},x_i,\dots ,x_1)
\in D_n^+(T\cup [n-i,n-1]),\]
where * denotes concatenation.
For example consider $\sigma=[2,-3,4,1,-5]\in D_5(\{1\})$ and $X=\{4,6,7\}$. 
Then $[8]\setminus X=\{1,2,3,5,8\}$ and 
$\sst^{-1}_{[8]\setminus X}(\sigma)=[2,-3,5,1,-8]$, thus
$\psi^{-1}(\sigma,X)=[2,-3,5,1,-8,7,6,4]\in D_8(\{1,6,7\})$.

Clearly $\psi^{-1}(\psi(\pi))=\pi$. Since $T\subseteq [0,n-i-2]$, we have $r_n(T\cup [n-i,n-1])=i$. From this 
it follows that $\psi(\psi^{-1}(\sigma,X))=(\sigma,X)$.

\end{proof}

\begin{lemma}\label{lemD1}
Let $D_n^+=\{\pi\in D_n:\pi(n)>0\}$. Then for $n\geq 1$
\[\sum_{\pi\in D_n^+}t^{\desD(\pi)}=Q_n(t)=\sum_{k=0}^{n-1}{n \choose k}D_k(t)(t-1)^{n-k-1}.\]
\end{lemma}

\begin{proof}

Given $U\subseteq D_n$ we have
\begin{align}
\sum_{\pi\in U}(t+1)^{\desD(\pi)} & =\sum_{\pi\in U}\sum_{j=0}^{\desD(\pi)}{\desD(\pi)\choose j}t^j \notag\\
&=\sum_{\pi\in U}\sum_{S\subseteq \DesD(\pi)}t^{|S|} \notag\\
&=\sum_{S\subseteq [0,n-1]}t^{|S|}\sum_{\pi\in U(S)}1\notag\\
&=\sum_{S\subseteq [0,n-1]}t^{|S|}|U(S)| \label{Ueq}.
\end{align}
Replacing $U$ with $D_n^+$ and using Lemma \ref{lemD0} we have
\begin{align*}
\sum_{\pi\in D_n^+}(t+1)^{\desD(\pi)}&
=\sum_{S\subseteq [0,n-1]}t^{|S|}|D_n^+(S)| \\
&=\sum_{S\subseteq [0,n-1]}t^{|S|}|D_{n-r_n(S)-1}(S\cap [0,n-r_n(S)-2])|{n \choose r_n(S)+1} \\
&=\sum_{i=0}^{n-1}\sum_{\substack{S\subseteq [0,n-1] \\ r_n(S)=i}}
t^{|S|}|D_{n-i-1}(S\cap [0,n-i-2])|{n \choose i+1} \\
&=\sum_{i=0}^{n-1}{n \choose i+1}t^i\sum_{\substack{S\subseteq [0,n-1] \\ r_n(S)=i}}
t^{|S|-i}|D_{n-i-1}(S\cap [0,n-i-2])|.
\end{align*}
Recall that if $r_n(S)=i$, then $S\supseteq [n-i,n-1]$ and $n-i-1\notin S$. Therefore each such $S$ can be expressed as $S=T\cup [n-i,n-1]$ for some $T\subseteq [0,n-i-2]$. Thus 
\begin{align*}
\sum_{\pi\in D_n^+}(t+1)^{\desD(\pi)}&
=\sum_{i=0}^{n-1}{n \choose i+1}t^i\sum_{T\subseteq [0,n-i-2]}t^{|T|}|D_{n-i-1}(T)|\\
&=\sum_{i=0}^{n-1}{n \choose i+1}t^i \sum_{\pi\in D_{n-i-1}}(t+1)^{\desD(\pi)}\\
&=\sum_{i=0}^{n-1}{n \choose i+1}t^i D_{n-i-1}(t+1)\\
&=\sum_{k=0}^{n-1}{n \choose k}t^{n-1-k} D_{k}(t+1),
\end{align*}
where the second equality follows from \eqref{Ueq}, and the last equality is obtained by setting $i=n-1-k$. 
Finally, the lemma follows by replacing $t$ with $t-1$.

\end{proof}

Next we provide the analog of Lemma \ref{lemB2}.

\begin{lemma}\label{lemD2}
Define $D_n^-=\{\pi\in D_n:\pi(n)<0\}$. Then for $n\geq 2$
\[\sum_{\pi\in D_n^-}t^{\desD(\pi)}=t^nQ_n(t^{-1}).\]
\end{lemma}

\begin{proof}

We want a bijection $\varphi:D_n^+\rightarrow D_n^-$. Note that we cannot flip all the signs as we did in the 
type B case, since we must ensure that $\varphi(\pi)$ has an even number of negative letters. So we define
\[\varphi(\pi)=\begin{cases}
[-\pi(1),-\pi(2),\dots ,-\pi(n)] & \text{if }n\text{ is even} \\
[\pi(1),-\pi(2),-\pi(3),\dots ,-\pi(n)] & \text{if }n\text{ is odd}\end{cases}\]

We claim that $\desD(\pi)+\desD(\varphi(\pi))=n$. The claim follows immediately if $n$ is even, 
just as in the type B case. So assume $n$ is odd,
in which case it is not true in general that $\DesD(\varphi(\pi))=[0,n-1]\setminus \DesD(\pi)$. However we still have $\pi(i)>\pi(i+1)$ if and only if $\varphi(\pi(i))<\varphi(\pi(i))$ for $i=2,3,\dots ,n-1$.
And the following table gives a case by case analysis to show that 
$|\DesD(\pi)\cap [0,1]|+|\DesD(\varphi(\pi))\cap [0,1]|=2$. Let $0<i<j$.

\

\[\begin{tabular}{c|c|c|c}
$\pi(1),\pi(2)$ & $\varphi(\pi(1)),\varphi(\pi(2))$ & $|\DesD(\pi)\cap [0,1]|$ & $|\DesD(\varphi\pi)\cap [0,1]|$ \\
\hline \hline
$i,j$ & $i,-j$ & 0 & 2 \\
\hline
$i,-j$ & $i,j$ & 2 & 0 \\
\hline
$-i,j$ & $-i,-j$ & 0 & 2 \\
\hline
$-i,-j$ & $-i,j$ & 2 & 0 \\
\hline
$j,i$ & $j,-i$ & 1 & 1 \\
\hline
$j,-i$ & $j,i$ & 1 & 1 \\
\hline
$-j,i$ & $-j,-i$ & 1 & 1 \\
\hline
$-j,-i$ & $-j,i$ & 1 & 1 \\
\hline
\end{tabular}\]

\

This verifies the claim. Therefore we have
\[\sum_{\pi\in D_n^-}t^{\desD(\pi)}=\sum_{\pi\in D_n^+}t^{\desD(\varphi(\pi))}
=\sum_{\pi\in D_n^+}t^{n-\desD(\pi)}
=t^n\sum_{k=0}^{n-1}t^{-\desD(\pi)}
=t^nQ_n(t^{-1}).\]

\end{proof}

\

\hspace{-12pt}\textbf{\textit{Proof of Theorem \ref{thmD}.}}
\begin{proof}

Using Lemmas \ref{lemD1} and \ref{lemD2} we have (a copy of the proof of Theorem \ref{thmB} replacing
B with D and $P$ with $Q$),
\[D_n(t)=\sum_{\pi\in D_n^+}t^{\desD(\pi)}+\sum_{\pi\in D_n^-}t^{\desD(\pi)}
=Q_n(t)+t^nQ_n(t^{-1}).\]

\end{proof}

\section{Real-rootedness of Eulerian polynomials}\label{real}

In examining the polynomials $P_n(t)$ and $Q_n(t)$ for small values of $n$, we observed
that $P_n(t)$ and $t^nP_n(t^{-1})$ have interlacing roots, and that 
$Q_n(t)$ and $t^nQ_n(t^{-1})$ also have interlacing roots. In this section we show that
this holds for general values of $n$, by explaining their connection to polynomials
studied by Savage and Visontai \cite{SavVis}.

We begin with some background on the real-rootedness of Eulerian polynomials.
A result first proved by Frobenius \cite{Frob} is that the type A Eulerian polynomials have only real roots.
Brenti \cite{Brenti} proved that the Eulerian polynomials of type B and of the exceptional finite
Coxeter groups have only real roots. Brenti also conjectured that for every finite Coxeter group, 
the corresponding Eulerian polynomials have only real roots. 
This conjecture was settled by Savage and Visontai \cite{SavVis} 
who showed that
the type D Eulerian polynomials do have only real roots. 
A $q$-analog of this result was proved by Yang and Zhang \cite{yz}. 
The proof in \cite{SavVis} includes an extension 
of techniques involving compatible polynomials. The notion of compatible polynomials
was introduced by Chudnovsky and Seymour \cite{ChuSey}, and is related to the idea of interlacing roots,
which we define next.

Let $f$ be a polynomial with real roots $\alpha_1\geq \alpha_2\geq \dots\geq \alpha_{\deg f}$, and let
$g$ be a polynomial with real roots $\beta_1\geq \beta_2\geq \dots\geq \beta_{\deg g}$. We say that
$f$ \textit{interlaces} $g$ if 
\[\dots\leq \alpha_2\leq\beta_2\leq\alpha_1\leq\beta_1.\]
Note that in this case we must have $\deg f\leq\deg g\leq 1+\deg f$. If $f$ interlaces $g$ or if 
$g$ interlaces $f$, then we also say that $f$ and $g$ have \textit{interlacing roots}.
The following remarkable theorem is due to Obreschkoff.
\begin{thm}[\cite{Obreschkoff}]\label{interlace}
Let $f,g\in\mathbb{R}[t]$ with $\deg f\leq\deg g\leq 1+\deg f$.
Then $f$ interlaces $g$ if and only if
$c_1f+c_2g$ has only real roots for all $c_1,c_2\in \mathbb{R}$.
\end{thm}

In light of Theorem \ref{interlace}, one may deduce from the recurrence (see \cite{Foata})
\[A_n(t)=(1+(n-1)t)A_{n-1}(t)+t(1-t)A_{n-1}'(t)\]
that the type A Eulerian polynomials have only real roots. And one may deduce from
the recurrence \cite{Brenti}
\[B_n(t)=((2n-1)t+1)B_{n-1}(t)+2t(1-t)B_{n-1}'(t)\]
that the type B Eulerian polynomials have only real roots.
Chow \cite[Theorem 5.3]{Chow2003} gave a recurrence for the type D Eulerian polynomials
involving the polynomials themselves and their derivatives. However the recurrence is quite complicated, and it does
not appear that the real-rootedness of the type D Eulerian polynomials can be easily deduced from this
recurrence. 

We turn our attention now to compatible polynomials. Call a set of polynomials 
$f_1,\dots ,f_k\in \mathbb{R}[t]$, \textit{compatible} if for all nonnegative numbers
$c_1,\dots ,c_k$ the polynomial $\sum_{i=1}^kc_if_i(t)$ has only real roots. Call a set of polynomials 
$f_1,\dots ,f_k\in \mathbb{R}[t]$, \textit{pairwise compatible} if for all $i,j\in [k]$ the polynomials
$f_i$ and $f_j$ are compatible. For polynomials with positive leading coefficients, Chudnovsky and Seymour 
showed that these two notions are equivalent.

\begin{thm}[{\cite[2.2]{ChuSey}}]\label{compat1}
Let $f_1,\dots ,f_k\in\mathbb{R}[t]$ be a set of pairwise compatible polynomials with positive leading coefficients.
Then $f_1,\dots ,f_k$ are compatible.
\end{thm}


Using the previous theorem, Savage and Visontai proved the following result, which is essential
to their work in \cite{SavVis}.

\begin{thm}[{\cite[Theorem 2.3]{SavVis}}]\label{compat2}
Let $f_1,\dots ,f_k\in\mathbb{R}[t]$ be a sequence of polynomials with positive leading coefficients such
that for all $1\leq i\leq j\leq k$ we have
\begin{enumerate}
\item[(a)] $f_i(t)$ and $f_j(t)$ are compatible, and
\item[(b)] $tf_i(t)$ and $f_j(t)$ are compatible.
\end{enumerate}
Given a sequence of integers $0\leq \lambda_0\leq \lambda_1\leq \dots \leq \lambda_m\leq k$ define
\[g_p(t)=\sum_{r=0}^{\lambda_p-1}tf_r(t)+\sum_{r=\lambda_p}^{k}f_p(t)\]
for $1\leq p\leq m$. Then for all $1\leq i\leq j\leq m$ we have
\begin{enumerate}
\item[(a')] $g_i(t)$ and $g_j(t)$ are compatible, and
\item[(b')] $tg_i(t)$ and $g_j(t)$ are compatible.
\end{enumerate}
\end{thm}

The following lemma is useful in connection with the previous theorem
(see \cite[Lemma 3.4]{Wag} and \cite[Lemma 2.5]{SavVis}).

\begin{lemma}\label{compat3}
Let $f,g\in\mathbb{R}[t]$ be polynomials with nonnegative coefficients. Then the following
two statements are equivalent:
\begin{enumerate}
\item[(i)] $f(t)$ and $g(t)$ are compatible, and $tf(t)$ and $g(t)$ are compatible.
\item[(ii)] $f$ interlaces $g$.
\end{enumerate}
\end{lemma}

In their proof that the type D Eulerian polynomials have only real roots, Savage and Visontai used
ascent sets of inversion sequences to construct a set of polynomials $T_{n,k}(t)$ for $0\leq k\leq 2n-1$,
and they used Theorem \ref{compat2} to show that these polynomials are compatible.
For $n\geq 2$ the involution on $B_n$ 
which changes the sign of the letter whose absolute value is one, is an involution which preserves type D 
descents (see \cite[Lemma 3.9]{SavVis}). One can combine this fact along with 
\cite[Theorem 3.12]{SavVis} to show that these polynomials may also be interpreted as follows
(see \cite{Branden} for a treatment of this topic that avoids inversion sequences).
\[T_{n,k}(t)=
\begin{cases}
\displaystyle 2\sum_{\substack{\pi\in D_n \\ \pi(n)=n-k}}t^{\desD(\pi)} & \text{ for }0\leq k\leq n-1\\
\displaystyle 2\sum_{\substack{\pi\in D_n \\ \pi(n)=n-1-k}}t^{\desD(\pi)} & \text{ for }n\leq k\leq 2n-1\\
\end{cases}.\]
In the proof of \cite[Theorem 3.15]{SavVis} the authors establish that for all $0\leq i<j\leq 2n-1$
the polynomials $T_{n,i}(t)$ and $T_{n,j}(t)$ are compatible, and the polynomials 
$tT_{n,i}(t)$ and $T_{n,j}(t)$ are also compatible, which implies the following.
\begin{thm}[\cite{SavVis}]\label{refine compat}

For $n\geq 4$ the set of polynomials 
\[T_{n,0}(t), T_{n,1}(t), \dots ,T_{n,2n-1}(t)\]
are compatible, and the set of polynomials
\[tT_{n,0}(t), tT_{n,1}(t), \dots ,tT_{n,n-1}(t),T_{n,n}(t), T_{n,n+1}(t), \dots ,T_{n,2n-1}(t)\] 
are also compatible.

\end{thm}

\begin{cor}\label{D conj}
For $n\geq 2$ the polynomial $Q_n(t)$ interlaces $t^nQ_n(t^{-1})$.
\end{cor}

\begin{proof}
For values of $n$ less than 4, this can be checked directly. For $n\geq 4$,
by Theorem \ref{refine compat} we have that for all $c_1,c_2\geq 0$ the polynomial
\[\frac{c_1}{2}\sum_{k=0}^{n-1}T_{n,k}(t)
+\frac{c_2}{2}\sum_{k=n}^{2n-1}T_{n,k}(t)=c_1Q_n(t)+c_2t^nQ(t^{-1})\]
has only real roots, and the polynomial
\[\frac{c_1}{2}\sum_{k=0}^{n-1}tT_{n,k}(t)
+\frac{c_2}{2}\sum_{k=n}^{2n-1}T_{n,k}(t)=c_1tQ_n(t)+c_2t^nQ(t^{-1})\]
has only real roots.
By Lemma \ref{compat3}, $Q_n(t)$ interlaces $t^nQ_n(t^{-1})$.
\end{proof}

In \cite[Corollory 3.13]{SavVis} the authors remark that \cite[Theorem 3.12]{SavVis}
can be used to show that $B_n(t)$ has only real roots. Indeed for $0\leq k\leq 2n-1$ 
one can use \cite[Theorem 3.12]{SavVis} to show that the following \textbf{s}-Eulerian
polynomials can be interpreted as follows
(again, see \cite{Branden} for a treatment that avoids inversion sequences).
\[E^{(2,4,\dots,2n)}_{n,k}(t)=\begin{cases}
\displaystyle \sum_{\substack{\pi\in B_n \\ \pi(n)=n-k}}t^{\desB(\pi)} & \text{ for }0\leq k\leq n-1\\
\displaystyle \sum_{\substack{\pi\in B_n \\ \pi(n)=n-1-k}}t^{\desB(\pi)} & \text{ for }n\leq k\leq 2n-1\\
\end{cases}.\]
For simplicity of notation, we define $B_{n,k}(t)=E^{(2,4,\dots,2n)}_{n,k}(t)$.




In their proof of \cite[Theorem 1.1]{SavVis}, the authors show that 
for all $0\leq i<j<s_n$
the polynomials $E_{n,i}^{(\textbf{s})}(t)$ and $E_{n,j}^{(\textbf{s})}(t)$ are compatible, 
and the polynomials 
$tE_{n,i}^{(\textbf{s})}(t)$ and $E_{n,j}^{(\textbf{s})}(t)$ are also compatible.
A special case is the following.


\begin{thm}[\cite{SavVis}]\label{refine compat B}
For $n\geq 1$ the set of polynomials
\[B_{n,0}(t), B_{n,1}(t), \dots ,B_{n,2n-1}(t)\]
are compatible, and the set of polynomials
\[tB_{n,0}(t), tB_{n,1}(t), \dots ,tB_{n,n-1}(t),B_{n,n}(t), B_{n,n+1}(t), \dots ,B_{n,2n-1}(t)\] 
are also compatible.
\end{thm}

In a manner analogous to Corollary \ref{D conj}, the previous theorem implies the following corollary.

\begin{cor}\label{B conj}
For $n\geq 1$ the polynomial $P_n(t)$ interlaces $t^nP_n(t^{-1})$.
\end{cor}

\begin{remark}
Recently Yang and Zhang \cite{yz2} gave a different proof of Corollary \ref{D conj} and Corollary
\ref{B conj}. Their methods include the Hermite-Biehler theorem, and a result of Borcea and Br\"and\'en on Hurwitz stability.

\end{remark}

\section{Acknowledgments}

The author would like to thank an anonymous referee for very useful suggestions that improved this paper.

\bibliographystyle{alpha}
\bibliography{my}

\end{document}